\documentclass[11pt]{article}
\usepackage{hyperref}
\usepackage{mathpazo}
\usepackage{amssymb,amsmath,amsthm}
\usepackage{epsfig}

\newtheorem{theorem}{Theorem}[section]

\newtheorem{definition}[theorem]{Definition}

\newtheorem{claim}[theorem]{Claim}
\newtheorem{lemma}[theorem]{Lemma}

\newtheorem{corollary}[theorem]{Corollary}

\newtheorem{remark}[theorem]{Remark}

\newcommand{\qedsymb}{\hfill{\rule{2mm}{2mm}}}
\renewenvironment{proof}[1][]{\begin{trivlist}
\item[\hspace{\labelsep}{\bf\noindent Proof#1:\/}] }{\qedsymb\end{trivlist}}

\def\calF{{\cal F}}

\def\calX{{\cal X}}
\def\calY{{\cal Y}}

\def\calI{{\cal I}}

\def\R{\mathbb{R}}
\def\S{\mathbb{S}}

\def\C{\mathbb{C}}

\newcommand{\cd}{\mathop{\mathrm{cd}}}

\newcommand\expectation{\mathop{{\mathbb{E}}}}

\def\lsoft{{l\kern-0.035cm\char39\kern-0.03truecm}}
\newcommand\dolnikov{Do{\lsoft}nikov}
\newcommand\kriz{Kriz}
\newcommand\barany{B\'ar\'any}

\newcommand{\Real}{\mathsf{Re}}
\newcommand{\Image}{\mathsf{Im}}

\newcommand{\eps}{\epsilon}
\renewcommand{\epsilon}{\varepsilon}

\newcommand{\rank}{\mathop{\mathrm{rank}}}
\newcommand{\minrank}{\mathop{\mathrm{minrk}}}

\newcommand{\Fset}{\mathbb{F}}         % The integers
\begin{document}

\title{{\bf Topological Bounds on the Dimension of Orthogonal Representations of Graphs}}

\author{
Ishay Haviv\thanks{School of Computer Science, The Academic College of Tel Aviv-Yaffo, Tel Aviv 61083, Israel.}
}

\date{}

\maketitle

\begin{abstract}
An orthogonal representation of a graph is an assignment of nonzero real vectors to its vertices such that distinct non-adjacent vertices are assigned to orthogonal vectors.
We prove general lower bounds on the dimension of orthogonal representations of graphs using the Borsuk-Ulam theorem from algebraic topology.
Our bounds strengthen the Kneser conjecture, proved by Lov\'asz in 1978, and some of its extensions due to \barany, Schrijver, \dolnikov, and \kriz.
As applications, we determine the integrality gap of fractional upper bounds on the Shannon capacity of graphs and the quantum one-round communication complexity of certain promise equality problems.
\end{abstract}

\section{Introduction}

A $t$-dimensional {\em orthogonal representation} of a graph $G=(V,E)$ over a field $\Fset$ is an assignment of a vector $u_i \in \Fset^t$ to each vertex $i \in V$ such that $\langle u_i,u_i \rangle \neq 0$ for every $i \in V$, and $\langle u_i,u_j \rangle = 0$ for every distinct non-adjacent vertices $i$ and $j$ in $G$. The {\em orthogonality dimension} of a graph $G$ over a field $\Fset$, denoted by $\xi_\Fset(G)$, is defined as the smallest integer $t$ for which there exists a $t$-dimensional orthogonal representation of $G$ over $\Fset$.
It is easy to verify that the orthogonality dimension of a graph is sandwiched between its independence number and its clique cover number, that is, for every graph $G$ and a field $\Fset$, $\alpha(G) \leq \xi_\Fset(G) \leq \chi(\overline{G})$.

The notion of orthogonal representations over the real field was introduced by Lov\'asz~\cite{Lovasz79} in the study of the Shannon capacity of graphs and was later involved in a geometric characterization of connectivity properties of graphs by Lov\'asz, Saks, and Schrijver~\cite{LovaszSS89}.
The orthogonality dimension over the complex field was used by de Wolf~\cite{deWolfThesis} in a characterization of the quantum one-round communication complexity of promise equality problems and by Cameron et al.~\cite{CameronMNSW07} in the study of the quantum chromatic number of graphs (see also~\cite{ScarpaS12,BrietBLPS15,BrietZ17}).
An extension of orthogonal representations, called orthogonal bi-representations, was introduced by Haemers~\cite{Haemers81} (see also~\cite{Peeters96}). Their smallest possible dimension, known as the minrank parameter of graphs, has found further applications in information theory, e.g.,~\cite{BBJK06,mazumdar2014duality,maleki2014index}, and in theoretical computer science, e.g.,~\cite{Valiant92,PudlakRS97,HavivL13} (see Section~\ref{sec:minrank}).

The present paper provides lower bounds on the orthogonality dimension of graphs over the real and complex fields using topological methods.
The use of topological methods in combinatorics was initiated in the study of the chromatic number of the Kneser graph defined as follows. For integers $d \geq 2s$, the Kneser graph $K(d,s)$ is the graph whose vertices are all the $s$-subsets of $[d]=\{1,\ldots,d\}$ where two sets are adjacent if they are disjoint.
In 1955, Kneser~\cite{Kneser55} observed that $K(d,s)$ admits a proper coloring with $d-2s+2$ colors, simply by coloring every set $A$ by the smallest integer in $A \cup \{d-2s+2\}$, and conjectured that fewer colors do not suffice.
In 1978, Lov\'asz~\cite{LovaszKneser} confirmed the conjecture by a breakthrough application of a tool from algebraic topology, the Borsuk-Ulam theorem~\cite{Borsuk33}.
Since then, topological methods have led to additional important results in combinatorics, discrete geometry, and theoretical computer science. For an in-depth background to the topic we refer the interested reader to Matou{\v{s}}ek's excellent book~\cite{MatousekBook}.

Following Lov\'asz's proof of the Kneser conjecture, several alternative proofs were given in the literature.
The simplest known proof of the conjecture is the one of Greene~\cite{Greene02}, inspired by a proof found by \barany~\cite{Barany78} soon after Lov\'asz's. Other proofs were given by \dolnikov~\cite{Dolnikov82}, Sarkaria~\cite{Sarkaria90}, and Matou{\v{s}}ek~\cite{Matousek04}, where Matou{\v{s}}ek's proof is the only one derived from a combinatorial argument (but the topological inspiration is still around).
Schrijver considered in~\cite{SchrijverKneser78} the graph $S(d,s)$ defined as the subgraph of $K(d,s)$ induced by the collection of all $s$-subsets of $[d]$ that include no consecutive integers modulo $d$ (that is, the $s$-subsets $A \subseteq [d]$ such that if $i \in A$ then $i+1 \notin A$, and if $d \in A$ then $1 \notin A$).
It was shown in~\cite{SchrijverKneser78} that $S(d,s)$ is a vertex-critical subgraph of $K(d,s)$, that is, its chromatic number is equal to that of $K(d,s)$ and a removal of any vertex of $S(d,s)$ decreases its chromatic number.

The various known proofs of the Kneser conjecture extend far beyond the chromatic number of the Kneser graph.
Extensions of these proofs to lower bounds on the chromatic number of general graphs were given by \dolnikov~\cite{Dolnikov82}, by \kriz~\cite{Kriz92}, and by Matou{\v{s}}ek and Ziegler~\cite{MatousekZ04} who generalized the proof techniques of Lov\'asz, Sarkaria, and \barany.
Such extensions are usually stated for a generalized Kneser graph $K(\calF)$, defined as the graph whose vertex set is a set system $\calF$ where two sets are adjacent if they are disjoint.
The generalized bounds are tight for the collection $\calF$ of all $s$-subsets of $[d]$ which corresponds to the graph $K(d,s)$, and some of them imply a tight lower bound on the chromatic number of the graph $S(d,s)$ as well.
It is not difficult to see that every graph is isomorphic to $K(\calF)$ for some set system $\calF$ (see, e.g.,~\cite{MatousekZ04}), hence the bounds hold for all graphs (but for certain graphs they are quite weak).
It was shown in~\cite{MatousekZ04} that the extensions of the proofs of Lov\'asz, Sarkaria, \barany, \dolnikov, and \kriz ~can be (almost) linearly ordered by strength, where Lov\'asz's original proof technique is the strongest.

\subsection{Topological Bounds on the Orthogonality Dimension}

We prove two general lower bounds on the orthogonality dimension of graphs over the real and complex fields and on the minrank parameter over the real field (see Definition~\ref{def:minrank}). For convenience, we state the results for the complements of the generalized Kneser graphs $K(\calF)$. As mentioned before, every graph can be represented in this form. The two bounds are proved using the Borsuk-Ulam theorem from algebraic topology.

The statement of our first bound is purely combinatorial. It strengthens the lower bounds on the chromatic number obtained independently by \dolnikov~\cite{Dolnikov82} and by \kriz~\cite{Kriz92}. Our proof is inspired by the proof of the Kneser conjecture by Greene~\cite{Greene02}.
To state the bound, we need the following definition (see, e.g.,~\cite[Section~3.4]{MatousekBook}).

\begin{definition}\label{def:cd2}
Let $\calF$ be a set system with ground set $[d]$ such that $\emptyset \notin \calF$.
The {\em $2$-colorability-defect} of $\calF$, denoted by $\cd_2(\calF)$, is the minimum size of a set $X \subseteq [d]$ such that the hypergraph on the vertex set $[d] \setminus X$ with the hyperedge set $\{ A \in \calF \mid A \cap X = \emptyset \}$ is $2$-colorable.
Equivalently, $\cd_2(\calF)$ is the minimum number of white elements in a coloring of $[d]$ by red, blue and white, such that no set of $\calF$ is completely red or completely blue (but it may be completely white).
\end{definition}

\begin{theorem}\label{thm:cd}
For every set system $\calF$ such that $\emptyset \notin \calF$,
\begin{enumerate}
  \item\label{thm:cd_itm:1} $\xi_\R(\overline{K(\calF)}) \geq \cd_2(\calF)$,
  \item\label{thm:cd_itm:2} $\xi_\C(\overline{K(\calF)}) \geq \frac{{\cd}_2(\calF)}{2}$, and
  \item\label{thm:cd_itm:3} ${\minrank}_\R(\overline{K(\calF)}) \geq \sqrt{\frac{{\cd}_2(\calF)}{2}}$.
\end{enumerate}
\end{theorem}
\noindent
As an easy consequence of Theorem~\ref{thm:cd}, we obtain that the orthogonality dimension of the complement of the Kneser graph $K(d,s)$ over the reals is $d-2s+2$ (see Corollary~\ref{cor:xiKneser}).

Our second bound has a geometric nature. Its proof employs the approach of \barany~\cite{Barany78} to the Kneser conjecture and strengthens a general lower bound on the chromatic number that follows from~\cite{Barany78} and is given explicitly in~\cite{MatousekZ04}.
In what follows, $\S^t$ stands for the $t$-dimensional unit sphere $\{ x \in \R^{t+1} \mid \|x\|=1 \}$, and an open hemisphere of $\S^t$ is a set of the form $\{ z \in \S^t \mid \langle x,z \rangle >0 \}$ for some $x \in \S^t$.

\begin{theorem}\label{thm:geo}
Let $\calF$ be a set system with ground set $[d]$ such that $\emptyset \notin \calF$.
Suppose that for an integer $t \geq 2$ there exist points $y_1,\ldots,y_d \in \S^{t-2}$ such that every open hemisphere of $\S^{t-2}$ contains the points of $\{y_i \mid i \in A\}$ for some $A \in \calF$.
Then,
\begin{enumerate}
  \item\label{thm:geom_itm:1} $\xi_\R(\overline{K(\calF)}) \geq t$,
  \item\label{thm:geom_itm:2} $\xi_\C(\overline{K(\calF)}) \geq \frac{t}{2}$, and
  \item\label{thm:geom_itm:3} ${\minrank}_\R(\overline{K(\calF)}) \geq \sqrt{\frac{t}{2}}$.
\end{enumerate}
\end{theorem}
\noindent
The theorem is used to prove that the orthogonality dimension of the complement of the Schrijver graph $S(d,s)$ over the reals is $d-2s+2$ (see Corollary~\ref{cor:xiSchrijver}).
The proof technique of Theorem~\ref{thm:geo} is also used to prove a lower bound on the orthogonality dimension over the reals of the complement of the Borsuk graph defined by Erd{\H{o}}s and Hajnal~\cite{ErdosH67} (see Section~\ref{sec:Borsuk}).

\subsection{Applications}

We describe below applications of our results to information theory and to quantum communication complexity.

\subsubsection{Shannon Capacity}\label{sec:Shannon}

The strong product $G_1 \cdot G_2$ of two graphs $G_1=(V_1,E_1)$ and $G_2=(V_2,E_2)$ is defined as the graph whose vertex set is $V_1 \times V_2$ where two distinct vertices $(u_1,u_2)$ and $(v_1,v_2)$ are adjacent if for every $i \in \{1,2\}$ the vertices $u_i$ and $v_i$ are either equal or adjacent in $G_i$. The $k$-th power of a graph $G$, denoted by $G^k$, is defined as the product of $k$ copies of $G$. The Shannon capacity of a graph $G$, introduced in 1956 by Shannon~\cite{Shannon56}, is the limit $c(G) = \lim_{k \rightarrow \infty}{(\alpha(G^k))^{1/k}}$ whose existence follows from super-multiplicativity and Fekete's lemma. This graph parameter is motivated by a question in information theory on the capacity of noisy channels.
Indeed, if $G$ is the graph whose vertices are the symbols that a channel can transmit, and two symbols are adjacent if they may be confused in the transmission, then $c(G)$ can be intuitively interpreted as the effective alphabet size of the channel.
The Shannon capacity parameter of graphs is very far from being well understood. It is not known if the problem of deciding whether the Shannon capacity of an input graph exceeds a given value is decidable, and the exact Shannon capacity is unknown even for small and fixed graphs, e.g., the cycle on $7$ vertices.

Several upper bounds on the Shannon capacity of graphs were presented in the literature over the years. It is easy to see that $c(G) \leq \chi(\overline{G})$, and Shannon showed already in~\cite{Shannon56} the stronger bound $c(G) \leq \chi_f(\overline{G})$, where $\chi_f$ stands for the fractional chromatic number.
A useful way to obtain an upper bound on the Shannon capacity is to come up with a real-valued non-negative sub-multiplicative function on graphs that forms an upper bound on the independence number, that is, a function $f$ satisfying $f(G_1 \cdot G_2) \leq f(G_1) \cdot f(G_2)$ for every two graphs $G_1,G_2$ and $\alpha(G) \leq f(G)$ for every graph $G$. Indeed, for such an $f$ we have
\[ c(G) = \lim_{k \rightarrow \infty}{(\alpha(G^k))^{1/k}} \leq \lim_{k \rightarrow \infty}{(f(G^k))^{1/k}} \leq \lim_{k \rightarrow \infty}{(f(G)^k)^{1/k}} = f(G).\]
For example, it is not difficult to verify that the orthogonality dimension $\xi_\Fset$ over any field $\Fset$ is a sub-multiplicative upper bound on the independence number, hence $c(G) \leq \xi_\Fset(G)$ for every graph $G$. Other upper bounds on the Shannon capacity obtained in this way are the $\vartheta$-function due to Lov\'asz~\cite{Lovasz79}, the ${\minrank}_\Fset$ parameter due to Haemers~\cite{Haemers81}, and the minimum dimension of polynomial representations due to Alon~\cite{AlonUnion98}.

In a recent work, Hu, Tamo, and Shayevitz~\cite{HuTS17} defined for every function $f$ as above a fractional linear programming variant $f^*$. For a graph $G$ on the vertex set $V$, $f^*(G)$ is the value of the following linear program.
\begin{align}
  \label{Primal}
  \begin{split}
 \textup{maximize } &\sum_{x \in V} w(x) \\
 \textup{subject to }  & \sum_{x\in S}w(x) \le f(G[S]) \textup{~~~ for each set $S \subseteq V$},\\
                & w(x)\ge 0 \textup{~~~ for each vertex $x \in V$},
  \end{split}
\end{align}
where $G[S]$ stands for the subgraph of $G$ induced by $S$.
It was proved in~\cite{HuTS17} that if $f$ is a sub-multiplicative upper bound on the independence number then so is $f^*$, hence $f^*$ also forms an upper bound on the Shannon capacity.
Moreover, the upper bound $f^*$ is at least as strong as $f$, that is, $c(G) \leq f^*(G) \leq f(G)$ for every graph $G$ (see Section~\ref{sec:Shannon4}).
It was shown in~\cite{HuTS17} that the Lov\'asz $\vartheta$-function satisfies $\vartheta(G)=\vartheta^*(G)$ for every graph $G$, whereas for other upper bounds $f$ on the Shannon capacity one can have $f^*(G) < f(G)$. For example, for every odd integer $n \geq 5$ the cycle $C_n$ satisfies ${\minrank}^*_\R(C_n) = \frac{n}{2} < \frac{n+1}{2} = {\minrank}_\R(C_n)$.

In this work, we aim to study the integrality gap of the fractional quantities $f^*$ as a function of the number of vertices. Namely, we would like to estimate the largest possible ratio $f(G)/f^*(G)$ over all $n$-vertex graphs $G$.
We start with a general upper bound.
A function $f$ on graphs is said to be sub-additive if for every graph $G$ on the vertex set $V$ and every sets $S_1$ and $S_2$ such that $V = S_1 \cup S_2$, $f(G) \leq f(G[S_1]) + f(G[S_2])$.
The following theorem shows that for sub-additive functions $f$ the ratio between $f$ and $f^*$ is at most logarithmic in the number of vertices.

\begin{theorem}\label{thm:fractional_upper}
For every sub-additive function $f$ and every $n$-vertex graph $G$,
\[f(G) \leq O(\log n) \cdot f^*(G).\]
\end{theorem}
\noindent
It is easy to verify that all the aforementioned upper bounds on the Shannon capacity are sub-additive, hence their fractional variants cannot improve the bound on the Shannon capacity by more than a logarithmic multiplicative term.

As an application of our results on the Kneser graph, we obtain a matching lower bound on the integrality gap of the fractional orthogonality dimension over the real and complex fields.

\begin{theorem}\label{thm:IntroIntegrality}
For every fixed $\eps >0$, there exists an explicit family of $n$-vertex graphs $G$ such that
\[\xi_\C^*(G) \leq \xi_\R^*(G) \leq 2+\eps  \mbox{~~whereas~~} \xi_\R (G)= \Theta(\log n) \mbox{~~and~~} \xi_\C (G)= \Theta(\log n).\]
%$\xi_\C^*(G) \leq \xi_\R^*(G) \leq 2+\eps$ whereas $\xi_\R (G)= \Theta(\log n)$ and $\xi_\C (G)= \Theta(\log n)$.
\end{theorem}
\noindent
We also show an unbounded integrality gap for the fractional minrank parameter over various fields (see Theorem~\ref{thm:AllBounds}).

\subsubsection{Quantum Communication Complexity}
In the standard model of communication complexity, two parties Alice and Bob get inputs $x,y$ from two sets $\calX,\calY$ respectively, and they have to compute by a communication protocol the value of $g(x,y)$ for a two-variable function $g$.
In a promise communication problem, the inputs are guaranteed to be drawn from a subset of $\calX \times \calY$ known to the parties in advance. In a one-round protocol, the communication flows only from Alice to Bob.
The classical, respectively quantum, communication complexity of a problem is the minimal number of bits, respectively qubits, that the parties have to exchange on worst-case inputs in a communication protocol for the problem.

The orthogonality dimension of graphs over the complex field plays a central role in the study of the quantum communication complexity of promise equality problems.\footnote{Note that the orthogonality dimension parameter (also known as orthogonality rank) is sometimes defined in the quantum communication complexity literature as the orthogonality dimension of the complement graph, namely, the definition requires vectors associated with {\em adjacent} vertices to be orthogonal. In this paper we have decided to follow the definition commonly used in the information theory literature.}
In such problems, Alice and Bob get either equal or adjacent vertices of a graph $G$ and their goal is to decide whether their inputs are equal. De Wolf~\cite[Section~8.5]{deWolfThesis} showed that the classical one-round communication complexity of the promise equality problem associated with a graph $G$ is $\lceil \log_2 \chi(G) \rceil$, and that its quantum one-round communication complexity is $\lceil \log_2 \xi_\C(\overline{G}) \rceil$. Bri{\"{e}}t et al.~\cite{BrietBLPS15} proved that any classical protocol for such a problem can always be reduced to a classical one-round protocol with no extra communication, while in the quantum setting the one-round and two-round communication complexities of a promise equality problem can have an exponential gap.
This separation was obtained using the Lov\'asz $\vartheta$-function and the relation $\xi_\C(G) \geq \vartheta(G)$ (see~\cite[Lemma~2.5]{BrietBLPS15}).

For a set system $\calF$ with $\emptyset \notin \calF$, consider the promise equality problem in which Alice and Bob get either equal or disjoint sets from $\calF$, and their goal is to decide whether their inputs are equal. Observe that the graph associated with this problem is the generalized Kneser graph $K(\calF)$, hence its quantum one-round communication complexity is precisely $\lceil \log_2 \xi_\C(\overline{K(\calF)}) \rceil$. Our bounds on the orthogonality dimension of such graphs over $\C$ have applications to the quantum one-round communication complexity of promise equality problems, as demonstrated below.

For integers $d \geq 2s$, consider the communication complexity problem in which Alice and Bob get two $s$-subsets of $[d]$, their sets are guaranteed to be either equal or disjoint, and their goal is to decide whether the sets are equal. The graph associated with this problem is the Kneser graph $K(d,s)$, so its classical communication complexity is $\lceil \log_2 \chi(K(d,s)) \rceil = \lceil \log_2 (d-2s+2) \rceil$. As an application of Theorem~\ref{thm:cd}, we get that
$\lceil \log_2 \xi_\C(\overline{K(d,s)}) \rceil \geq \lceil \log_2 (d-2s+2)-1 \rceil$ (see Corollary~\ref{cor:xiKneser}),
yielding the precise quantum one-round communication complexity of the problem up to an additive $1$.
We note that the lower bound on $\xi_\C(\overline{K(d,s)})$ obtained from the Lov\'asz $\vartheta$-function would not suffice here, since $\vartheta(\overline{K(d,s)}) = \frac{d}{s}$ (see~\cite{Lovasz79}).

The orthogonality dimension over the complex field was also used by Bri{\"{e}}t et al.~\cite{BrietBLPS15} to characterize the quantum one-round communication complexity of a family of problems called list problems, originally studied by Witsenhausen~\cite{Witsenhausen76}. In the list problem that corresponds to the Kneser graph $K(d,s)$, Alice gets an $s$-subset $A$ of $[d]$, Bob gets a list of pairwise disjoint $s$-subsets of $[d]$ that includes the set $A$, and his goal is to discover $A$. It follows from~\cite{BrietBLPS15} that the quantum one-round communication complexity of this problem is equal to the quantity $\lceil \log_2 \xi_\C(\overline{K(d,s)}) \rceil$ determined above.

\subsection{Outline}
The rest of the paper is organized as follows.
In Section~\ref{sec:preliminaries} we provide some background on the Borsuk-Ulam theorem and on the minrank parameter of graphs.
In Section~\ref{sec:bounds} we prove Theorems~\ref{thm:cd} and~\ref{thm:geo} and obtain our results on the Kneser, Schrijver, and Borsuk graphs.
Finally, in Section~\ref{sec:Shannon4} we study the integrality gap of fractional upper bounds on the Shannon capacity and prove Theorems~\ref{thm:fractional_upper} and~\ref{thm:IntroIntegrality}.

\section{Preliminaries}\label{sec:preliminaries}

\subsection{The Borsuk-Ulam Theorem}\label{sec:BU}

For an integer $t \geq 0$, let $\S^{t} = \{ x \in \R^{t+1} \mid \|x\|=1 \}$ denote the $t$-dimensional unit Euclidean sphere.
For a point $x \in \S^t$, let $H(x) = \{ z \in \S^t \mid \langle x,z \rangle >0 \}$ denote the open hemisphere of $\S^t$ centered at $x$.
We state below the Borsuk-Ulam theorem, proved by Borsuk in 1933~\cite{Borsuk33}.

\begin{theorem}[Borsuk-Ulam Theorem]\label{thm:BU}
For every continuous function $f: \S^t \rightarrow \R^t$ there exists $x \in \S^t$ such that $f(x)=f(-x)$.
\end{theorem}
\noindent
Equivalently, if a continuous function $f: \S^t \rightarrow \R^{t'}$ satisfies $f(x) \neq f(-x)$ for all $x \in \S^t$ then $t' > t$.
For several other equivalent versions of Theorem ~\ref{thm:BU}, see~\cite[Section~2.1]{MatousekBook}.

We also need a variant of the Borsuk-Ulam theorem for complex-valued functions.
We start with some notations.
For a complex number $z \in \C$ we denote by $\Real(z)$ and $\Image(z)$ the real and imaginary parts of $z$ respectively, hence $z = \Real(z)+\Image(z) \cdot i$.
For an integer $t \geq 1$, let $\phi_t:\C^t \rightarrow \R^{2t}$ be the natural embedding of $\C^t$ in $\R^{2t}$ defined by
\begin{eqnarray}\label{eq:phi}
\phi_t(x) = (\Real(x_1),\Image(x_1),\ldots,\Real(x_t),\Image(x_t)) \in \R^{2t}.
\end{eqnarray}
Clearly, $\phi_t$ is a bijection from $\C^t$ to $\R^{2t}$.
Notice that we have $\phi_t(-x)=-\phi_t(x)$ for every $x \in \C^t$.
Our variant of the Borsuk-Ulam theorem for complex-valued functions is given below.

\begin{theorem}[Borsuk-Ulam Theorem for complex-valued functions]\label{thm:BU_complex}
For every continuous function $f: \S^{2t} \rightarrow \C^t$ there exists $x \in \S^{2t}$ such that $f(x)=f(-x)$.
\end{theorem}

\begin{proof}
For a continuous function $f: \S^{2t} \rightarrow \C^t$ consider the function $\widetilde{f}: \S^{2t} \rightarrow \R^{2t}$ defined by the composition $\widetilde{f} = \phi_t \circ f$. Applying Theorem~\ref{thm:BU} to $\widetilde{f}$, we get that there exists $x \in \S^{2t}$ such that $\widetilde{f}(x)=\widetilde{f}(-x)$. By the invertibility of $\phi_t$, this implies that $f(x) = f(-x)$, as required.
\end{proof}

\subsection{Minrank}\label{sec:minrank}

The minrank parameter of graphs, introduced by Haemers in~\cite{Haemers81}, is defined as follows.
\begin{definition}\label{def:minrank}
Let $G$ be a graph on the vertex set $V = [n]$ and let $\Fset$ be a field.
We say that an $n \times n$ matrix $M$ over $\Fset$ {\em represents} $G$ if $M_{i,i} \neq 0$ for every $i \in V$, and $M_{i,j}=0$ for every distinct non-adjacent vertices $i,j \in V$.
The {\em minrank} of $G$ over $\Fset$ is defined as
\[{\minrank}_\Fset(G) =  \min\{{\rank}_{\Fset}(M)\mid M \mbox{ represents }G\mbox{ over }\Fset\}.\]
\end{definition}

The minrank parameter can be equivalently defined in terms of orthogonal bi-representations. A $t$-dimensional {\em orthogonal bi-representation} of a graph $G=(V,E)$ over a field $\Fset$ is an assignment of a pair $(u_i,v_i) \in \Fset^t \times \Fset^t$ to each vertex $i \in V$ such that $\langle u_i,v_i \rangle \neq 0$ for every $i \in V$, and  $\langle u_i,v_j \rangle = \langle u_j,v_i \rangle = 0$ for every distinct non-adjacent vertices $i$ and $j$ in $G$. It can be verified that ${\minrank}_\Fset(G)$ is the smallest integer $t$ for which there exists a $t$-dimensional orthogonal bi-representation of $G$ over $\Fset$ (see, e.g.,~\cite{Peeters96,ChlamtacH14}). Since orthogonal bi-representations generalize orthogonal representations, we clearly have ${\minrank}_\Fset(G) \leq \xi_\Fset(G)$ for all graphs $G$ and fields $\Fset$.

The minrank parameter is always bounded from above by the clique cover number.
The following lemma shows a lower bound on the minrank parameter over finite fields in terms of the clique cover number.
Its proof is implicit in~\cite{LangbergS08} and we give here a quick proof for completeness.

\begin{lemma}[\cite{LangbergS08}]\label{claim:minrankF}
For every graph $G$ and a finite field $\Fset$, ${\minrank}_\Fset(G) \geq \log_{|\Fset|}\chi(\overline{G})$.
\end{lemma}

\begin{proof}
Denote $t = {\minrank}_\Fset(G)$. Then there exists an assignment of a pair $(u_i,v_i) \in \Fset^t \times \Fset^t$ to each vertex $i \in V$ that forms an orthogonal bi-representation of $G$ over $\Fset$. Consider the coloring that assigns to every vertex $i \in V$ the color $u_i \in \Fset^t$. We claim that this is a proper coloring of $\overline{G}$. Indeed, for two vertices $i$ and $j$ adjacent in $\overline{G}$ we have $\langle u_i,v_i \rangle \neq 0$ and $\langle u_j,v_i \rangle = 0$, hence $u_i \neq u_j$. Since the number of used colors is at most $|\Fset|^t$, it follows that $\chi(\overline{G}) \leq |\Fset|^t$, as required.
\end{proof}

\section{Topological Bounds on the Orthogonality Dimension}\label{sec:bounds}

In this section we prove Theorems~\ref{thm:cd} and~\ref{thm:geo} and obtain our results on the Kneser, Schrijver, and Borsuk graphs.
The proofs employ the Borsuk-Ulam theorem given in Section~\ref{sec:BU}.

\subsection{Proof of Theorem~\ref{thm:cd}}

We start with the lower bound on the orthogonality dimension over the real field.

\begin{proof}[ of Theorem~\ref{thm:cd}, Item~\ref{thm:cd_itm:1}]
Let $\calF$ be a set system with ground set $[d]$ such that $\emptyset \notin \calF$, and put $t = \xi_\R(\overline{K(\calF)})$.
Then there exists an assignment of a nonzero vector $u_A \in \R^t$ to every set $A \in \calF$, such that $\langle u_A,u_B \rangle = 0$ for every disjoint sets $A,B \in \calF$. It can be assumed without loss of generality that the first nonzero coordinate in every vector $u_A$ is positive (otherwise replace $u_A$ by $-u_A$).

Let $y_1,\ldots,y_d \in \S^t$ be $d$ points in a general position, that is, no $t+1$ of them lie on a $(t-1)$-dimensional sphere. For a set $A \subseteq [d]$ denote $y_A = \{ y_i \mid i \in A \}$. Define a function $f:\S^t \rightarrow \R^t$ by
\[ f(x) = \sum_{A \in \calF}{u_A \cdot \prod_{j \in A}{\max(\langle x, y_j\rangle,0)}}.\]
Observe that for every $x \in \S^{t}$, $f(x)$ is a linear combination with positive coefficients of the vectors $u_A$ such that $y_A \subseteq H(x)$, where $H(x)$ is the open hemisphere of $\S^t$ centered at $x$.
The function $f$ is clearly continuous, hence by Theorem~\ref{thm:BU} there exists $x \in \S^t$ such that $f(x) = f(-x)$.
However, $f(x)$ is a linear combination of the vectors $u_A$ with $y_A \subseteq H(x)$ whereas $f(-x)$ is a linear combination of the vectors $u_A$ with $y_A \subseteq H(-x)$. Since $H(x) \cap H(-x) = \emptyset$, it follows that the sets $A$ involved in the linear combination of $f(x)$ are all disjoint from those involved in the linear combination of $f(-x)$.
The fact that $\langle u_A,u_B \rangle = 0$ for every disjoint sets $A,B \in \calF$ yields that the vectors $f(x)$ and $f(-x)$ are orthogonal, and by $f(x) = f(-x)$ they must be equal to the zero vector.

We claim now that there is no $A \in \calF$ with $y_A \subseteq H(x)$. To see this, assume in contradiction that $A_1, \ldots, A_m$ are the sets with this property ($m \geq 1$), and let $j \in [t]$ be the least coordinate in which at least one of the vectors $u_{A_1},\ldots,u_{A_m}$ is nonzero. Since $f(x)$ is a linear combination of these vectors with positive coefficients, using the assumption that the first nonzero coordinate of every $u_{A}$ is positive, it follows that the $j$th coordinate of $f(x)$ is positive in contradiction to $f(x)$ being the zero vector.
By the same reasoning, there is no $A \in \calF$ with $y_A \subseteq H(-x)$.

Finally, let $X$ denote the set of indices $i \in [d]$ for which $y_i$ does not belong to $H(x)$ nor to $H(-x)$. By the assumption of general position, $|X| \leq t$. We color the elements of $[d] \setminus X$ as follows: If $y_i \in H(x)$ then $i$ is colored red, and if $y_i \in H(-x)$ then $i$ is colored blue. Since no set $A \in \calF$ satisfies $y_A \subseteq H(x)$ or $y_A \subseteq H(-x)$, we get a proper $2$-coloring of the hypergraph $([d] \setminus X, \{ A \in \calF \mid A \cap X = \emptyset \})$. This implies that $\cd_2(\calF) \leq t$, as required.
\end{proof}

We next prove our lower bound on the orthogonality dimension over the complex field.
Its proof is similar to the proof over the reals but requires the Borsuk-Ulam theorem for complex-valued functions (Theorem~\ref{thm:BU_complex}).
Recall that $\phi_t$ stands for the natural embedding of $\C^t$ in $\R^{2t}$ given in~\eqref{eq:phi}.

\begin{proof}[ of Theorem~\ref{thm:cd}, Item~\ref{thm:cd_itm:2}]
Let $\calF$ be a set system with ground set $[d]$ such that $\emptyset \notin \calF$, and put $t = \xi_\C(\overline{K(\calF)})$.
Then there exists an assignment of a nonzero vector $u_A \in \C^t$ to every set $A \in \calF$, such that $\langle u_A,u_B \rangle = 0$ for every disjoint sets $A,B \in \calF$. It can be assumed without loss of generality that the first nonzero coordinate in every vector $\phi_t(u_A)$ is positive (otherwise replace $u_A$ by $-u_A$).

Let $y_1,\ldots,y_d \in \S^{2t}$ be $d$ points in a general position.
As before, for a set $A \subseteq [d]$ denote $y_A = \{ y_i \mid i \in A \}$. Define a function $f:\S^{2t} \rightarrow \C^t$ by
\[ f(x) = \sum_{A \in \calF}{u_A \cdot \prod_{j \in A}{\max(\langle x, y_j\rangle,0)}}.\]
Observe that for every $x \in \S^{2t}$, $f(x)$ is a linear combination with real positive coefficients of the vectors $u_A$ satisfying $y_A \subseteq H(x)$.
The function $f$ is clearly continuous, hence by Theorem~\ref{thm:BU_complex} there exists $x \in \S^{2t}$ such that $f(x) = f(-x)$.
However, $f(x)$ is a linear combination of the vectors $u_A$ with $y_A \subseteq H(x)$ whereas $f(-x)$ is a linear combination of the vectors $u_A$ with $y_A \subseteq H(-x)$. Since $H(x) \cap H(-x) = \emptyset$, it follows that the sets $A$ involved in the linear combination of $f(x)$ are all disjoint from those involved in the linear combination of $f(-x)$.
The fact that $\langle u_A,u_B \rangle = 0$ for every disjoint sets $A,B \in \calF$ yields that the vectors $f(x)$ and $f(-x)$ are orthogonal, and by $f(x) = f(-x)$ they must be equal to the zero vector.

We claim now that there is no $A \in \calF$ with $y_A \subseteq H(x)$. To see this, assume in contradiction that $A_1, \ldots, A_m$ are the sets with this property ($m \geq 1$), and let $j \in [2t]$ be the least coordinate in which at least one of the vectors $\phi_{t}(u_{A_1}),\ldots,\phi_{t}(u_{A_m})$ is nonzero. It follows that the $j$th coordinate of $\phi_t(f(x))$ is positive in contradiction to $f(x)$ being the zero vector.
By the same reasoning, there is no $A \in \calF$ with $y_A \subseteq H(-x)$.

Finally, let $X$ denote the set of indices $i \in [d]$ for which $y_i$ does not belong to $H(x)$ nor to $H(-x)$. By the assumption of general position, $|X| \leq 2t$. We color the elements of $[d] \setminus X$ as follows: If $y_i \in H(x)$ then $i$ is colored red, and if $y_i \in H(-x)$ then $i$ is colored blue. Since no set $A \in \calF$ satisfies $y_A \subseteq H(x)$ or $y_A \subseteq H(-x)$, we get a proper $2$-coloring of the hypergraph $([d] \setminus X, \{ A \in \calF \mid A \cap X = \emptyset \})$. This implies that $\cd_2(\calF) \leq 2t$, as required.
\end{proof}

Finally, we prove our lower bound on the minrank parameter over the real field (recall Definition~\ref{def:minrank}).
We start with the following lemma. Here, a real matrix is said to be non-negative if all of its entries are non-negative.

\begin{lemma}\label{lemma:minrank_nonneg}
Let $\calF$ be a set system such that $\emptyset \notin \calF$, and let $M$ be a real non-negative matrix that represents the graph $\overline{K(\calF)}$ over $\R$. Then, ${\rank}_\R(M) \geq \frac{1}{2} \cdot \cd_2(\calF)$.
\end{lemma}

\begin{proof}
Let $\calF$ be a set system of size $n$ with ground set $[d]$ such that $\emptyset \notin \calF$, let $M \in \R^{n \times n}$ be a non-negative matrix that represents the graph $\overline{K(\calF)}$ over $\R$, and put $t = {\rank}_\R(M)$. Write $M = M_1^T \cdot M_2$ for matrices $M_1,M_2 \in \R^{t \times n}$. For every set $A \in \calF$, let $u_A$ and $v_A$ be the $t$-dimensional columns associated with $A$ in $M_1$ and $M_2$ respectively, and let $w_A = u_A \circ v_A$ be the $2t$-dimensional concatenation of $u_A$ and $v_A$.
Since $M$ represents $\overline{K(\calF)}$, we have $\langle u_A,v_A \rangle \neq 0$ for every $A \in \calF$, and $\langle u_A,v_B \rangle = 0$ for every disjoint sets $A,B \in \calF$. By the assumption that $M$ is non-negative, we also have $\langle u_A,v_B \rangle \geq 0$ for all $A,B \in \calF$.

Let $y_1,\ldots,y_d \in \S^{2t}$ be $d$ points in a general position. As before, for a set $A \subseteq [d]$ denote $y_A = \{ y_i \mid i \in A \}$. Define a function $f:\S^{2t} \rightarrow \R^{2t}$ by
\[ f(x) = \sum_{A \in \calF}{w_A \cdot \prod_{j \in A}{\max(\langle x, y_j\rangle,0)}}.\]
Observe that for every $x \in \S^{2t}$, $f(x)$ is a linear combination with positive coefficients of the vectors $w_A$ such that $y_A \subseteq H(x)$.
The function $f$ is clearly continuous, hence by Theorem~\ref{thm:BU} there exists $x \in \S^{2t}$ such that $f(x) = f(-x)$.
For this $x$, denote $f(x) = f(-x) = w = w_1 \circ w_2$ where $w_1,w_2 \in \R^t$.
By $f(x)=w$ we get that $w_1$ is a linear combination of the vectors $u_A$ with $y_A \subseteq H(x)$, and by $f(-x)=w$ we get that $w_2$ is a linear combination of the vectors $v_A$ with $y_A \subseteq H(-x)$.
However, the sets $A$ with $y_A \subseteq H(x)$ are all disjoint form the sets $A$ with $y_A \subseteq H(-x)$, hence
the fact that $\langle u_A,v_B \rangle = 0$ for every disjoint sets $A,B \in \calF$ yields that the vectors $w_1$ and $w_2$ are orthogonal.

We claim now that there is no $A \in \calF$ with $y_A \subseteq H(x)$. To see this, assume in contradiction that $A_1, \ldots, A_m$ are the sets with this property ($m \geq 1$). By the definition of $f$ we can write $w = f(x)= \sum_{i=1}^{m}{c_i \cdot w_{A_i}}$ for some positive coefficients $c_i>0$. However, this implies that
\[ \langle w_1,w_2\rangle = \langle \sum_{i=1}^{m}{c_i \cdot u_{A_i}} ~, \sum_{i=1}^{m}{c_i \cdot v_{A_i}}\rangle = \sum_{1 \leq i,j \leq m}{c_i c_j \cdot \langle u_{A_i},v_{A_j} \rangle} > 0, \]
where the inequality holds since $\langle u_{A_i},v_{A_j} \rangle \geq 0$ for all pairs $i,j$ and $\langle u_{A_i},v_{A_i} \rangle > 0$ for every $i$.
This is in contradiction to the fact that the vectors $w_1$ and $w_2$ are orthogonal.
By the same reasoning, there is no $A \in \calF$ with $y_A \subseteq H(-x)$.

Finally, let $X$ denote the set of indices $i \in [d]$ for which $y_i$ does not belong to $H(x)$ nor to $H(-x)$. By the assumption of general position, $|X| \leq 2t$. We color the elements of $[d] \setminus X$ as follows: If $y_i \in H(x)$ then $i$ is colored red, and if $y_i \in H(-x)$ then $i$ is colored blue. Since no set $A \in \calF$ satisfies $y_A \subseteq H(x)$ or $y_A \subseteq H(-x)$, we get a proper $2$-coloring of the hypergraph $([d] \setminus X, \{ A \in \calF \mid A \cap X = \emptyset \})$. This implies that $\cd_2(\calF) \leq 2t$, as required.
\end{proof}

Equipped with Lemma~\ref{lemma:minrank_nonneg}, we are ready to complete the proof of Theorem~\ref{thm:cd}.

\begin{proof}[ of Theorem~\ref{thm:cd}, Item~\ref{thm:cd_itm:3}]
Let $\calF$ be a set system of size $n$ such that $\emptyset \notin \calF$, and let $M$ be an $n \times n$ matrix that represents the graph $\overline{K(\calF)}$ over $\R$.
Consider the $n \times n$ matrix $M'$ defined by $M'_{i,j} = M_{i,j}^2$ for all $i,j$.
It is well known and easy to check that $M'$ is a principal sub-matrix of the tensor product $M \otimes M$ of $M$ with itself, hence
\[ {\rank}_\R (M') \leq {\rank}_\R (M \otimes M) = {\rank}_\R (M)^2.\]
The non-negative matrix $M'$ represents $\overline{K(\calF)}$ since it has the same zero pattern as $M$, so we can apply Lemma~\ref{lemma:minrank_nonneg} to obtain that
\[ {\rank}_\R(M) \geq \sqrt{{\rank}_\R(M')} \geq \sqrt{\frac{{\cd}_2(\calF)}{2}},\]
completing the proof.
\end{proof}

\subsubsection{The Kneser Graph}\label{sec:Kneser}
Recall that for integers $d \geq 2s$, the Kneser graph $K(d,s)$ is the graph whose vertices are all the $s$-subsets of $[d]$, where two sets are adjacent if they are disjoint.
We need the following simple claim (see, e.g.,~\cite[Section~3.4]{MatousekBook}).

\begin{claim}\label{claim:cd2_Kneser}
For integers $d \geq 2s$, let $\calF$ be the collection of all $s$-subsets of $[d]$. Then, $\cd_2(\calF) = d-2s+2$.
\end{claim}

\begin{proof}
Let $X \subseteq [d]$ be an arbitrary set of size $d-2s+2$, and consider an arbitrary balanced $2$-coloring of the $2s-2$ elements of $[d] \setminus X$. Clearly, no $s$-subset of $[d] \setminus X$ is monochromatic, hence $\cd_2(\calF) \leq d-2s+2$.
For the other direction, notice that for every $X \subseteq [d]$ of size at most $d-2s+1$ there are at least $2s-1$ elements in $[d] \setminus X$, hence every $2$-coloring of $[d] \setminus X$ includes a monochromatic $s$-subset. This implies that $\cd_2(\calF) \geq d-2s+2$ and completes the proof.
\end{proof}

The following corollary summarizes our bounds for the Kneser graph.

\begin{corollary}\label{cor:xiKneser}
For every integers $d \geq 2s$,
\begin{enumerate}
  \item\label{itm:Kneser} $\xi_\R(\overline{K(d,s)}) = d-2s+2$,
  \item $\xi_\C(\overline{K(d,s)}) \geq (d-2s+2)/2$, and
  \item ${\minrank}_\R(\overline{K(d,s)}) \geq \sqrt{(d-2s+2)/2}$.
\end{enumerate}
\end{corollary}

\begin{proof}
Notice that $K(d,s)$ is the graph $K(\calF)$ where $\calF$ is the collection of all $s$-subsets of $[d]$.
The three lower bounds follow directly by combining Theorem~\ref{thm:cd} with Claim~\ref{claim:cd2_Kneser}.
The matching upper bound in Item~\ref{itm:Kneser} follows by $\xi_\R(\overline{K(d,s)}) \leq \chi(K(d,s)) = d-2s+2$.
%Note that the same upper bound holds also for $\xi_\C(\overline{K(d,s)})$ and for ${\minrank}_\R (\overline{K(d,s)})$.
\end{proof}

\subsection{Proof of Theorem~\ref{thm:geo}}

We prove below Item~\ref{thm:geom_itm:1} of Theorem~\ref{thm:geo}. The other two items follow similarly, using ideas from the proofs of Items~\ref{thm:cd_itm:2} and~\ref{thm:cd_itm:3} of Theorem~\ref{thm:cd}. To avoid repetitions, we omit the details.

\begin{proof}[ of Theorem~\ref{thm:geo}, Item~\ref{thm:geom_itm:1}]
Let $\calF$ be a set system with ground set $[d]$ such that $\emptyset \notin \calF$, and let $y_1,\ldots,y_d \in \S^{t-2}$ be the points given in the theorem.
Put $t' = \xi_\R(\overline{K(\calF)})$.
Then there exists an assignment of a nonzero vector $u_A \in \R^{t'}$ to every set $A \in \calF$, such that $\langle u_A,u_B \rangle = 0$ for every disjoint sets $A,B \in \calF$. It can be assumed without loss of generality that the first nonzero coordinate in every vector $u_A$ is positive (otherwise replace $u_A$ by $-u_A$).

Define a function $f:\S^{t-2} \rightarrow \R^{t'}$ by
\[ f(x) = \sum_{A \in \calF}{u_A \cdot \prod_{j \in A}{\max(\langle x, y_j\rangle,0)}}.\]
For any $x \in \S^{t-2}$, let $C_x$ be the collection of sets $A \in \calF$ such that $y_A \subseteq H(x)$, where, as before, $y_A = \{ y_i \mid i \in A \}$. By the assumption on the points $y_1,\ldots,y_d$ we have $|C_x| \geq 1$.
Observe that $f(x)$ is a linear combination with positive coefficients of the vectors $u_A$ with $A \in C_x$.
Letting $j \in [t']$ be the least coordinate in which at least one of the vectors of $\{ u_{A} \mid  A \in C_x \}$ is nonzero, it follows that the $j$th coordinate of $f(x)$ is positive, hence $f(x)$ is nonzero.
Further, by $H(x) \cap H(-x) = \emptyset$ we get that the sets of $C_x$ are all disjoint from those of $C_{-x}$, hence the fact that $\langle u_A,u_B \rangle = 0$ for every disjoint sets $A,B \in \calF$ yields that the vectors $f(x)$ and $f(-x)$ are orthogonal.

Consider the function $g:\S^{t-2} \rightarrow \S^{t'-1}$ defined by
\[g(x) = \frac{f(x)}{\|f(x)\|}.\]
Note that $g$ is well defined as $f(x)$ is nonzero for every $x \in \S^{t-2}$. Consider also the function $\widetilde{g} : \S^{t-2} \rightarrow \R^{t'-1}$ that maps every $x \in \S^{t-2}$ to the projection of $g(x)$ to its last $t'-1$ coordinates (i.e., all of its coordinates besides the first one). We claim that there is no $x \in \S^{t-2}$ such that $\widetilde{g}(x)=\widetilde{g}(-x)$. To see this, notice that $g(x)$ is a unit vector whose first entry is non-negative, so the projection of $g(x)$ to its last $t'-1$ coordinates fully determines $g(x)$.
This implies that if there exists an $x \in \S^{t-2}$ satisfying $\widetilde{g}(x)=\widetilde{g}(-x)$ then this $x$ also satisfies $g(x)=g(-x)$, in contradiction to the orthogonality of $f(x)$ and $f(-x)$. Since $\widetilde{g}$ is continuous we can apply Theorem~\ref{thm:BU} to derive that $t'-1 > t-2$ which implies that $t' \geq t$ and completes the proof.
\end{proof}

\subsubsection{The Schrijver Graph}

We say that a set $A \subseteq [d]$ is {\em stable} if it does not contain two consecutive elements modulo $d$ (that is, if $i \in A$ then $i+1 \notin A$, and if $d \in A$ then $1 \notin A$). In other words, a stable subset of $[d]$ is an independent set in the cycle $C_d$ with the numbering from $1$ to $d$ along the cycle.
Recall that for $d \geq 2s$, the Schrijver graph $S(d,s)$ is the graph whose vertices are all the stable $s$-subsets of $[d]$, where two sets are adjacent if they are disjoint.

We need the following strengthening of a lemma of Gale~\cite{Gale56} proved by Schrijver in~\cite{SchrijverKneser78}.
See~\cite[Section~3.5]{MatousekBook} for a nice proof by Ziegler based on the moment curve.

\begin{lemma}[\cite{SchrijverKneser78}]\label{lemma:Gale}
For every integers $d \geq 2s$, there exist points $y_1,\ldots,y_d \in \S^{d-2s}$ such that every open hemisphere of $\S^{d-2s}$ contains the points of $\{ y_i \mid i \in A \}$ for some stable $s$-subset $A$ of $[d]$.
\end{lemma}

For $d \geq 2s$, consider the collection $\calF$ of all stable $s$-subsets of $[d]$, and notice that $S(d,s)$ is the graph $K(\calF)$.
By Lemma~\ref{lemma:Gale}, $\calF$ satisfies the condition of Theorem~\ref{thm:geo} for $t=d-2s+2$.
This directly implies the following corollary which summarizes our bounds for the Schrijver graph.

\begin{corollary}\label{cor:xiSchrijver}
For every integers $d \geq 2s$,
\begin{enumerate}
  \item\label{itm:Kneser} $\xi_\R(\overline{S(d,s)}) = d-2s+2$,
  \item $\xi_\C(\overline{S(d,s)}) \geq (d-2s+2)/2$, and
  \item ${\minrank}_\R(\overline{S(d,s)}) \geq \sqrt{(d-2s+2)/2}$.
\end{enumerate}
\end{corollary}

\begin{remark}
We note that the bounds that Theorem~\ref{thm:cd} implies for the Schrijver graph $S(d,s)$ are weaker than the bounds obtained above. Indeed, it is easy to check that the set system $\calF$ that corresponds to the graph $S(d,s)$ satisfies $\cd_2(\calF) = d-4s+4$.
For a discussion comparing the bounds derived from Theorems~\ref{thm:cd} and~\ref{thm:geo}, see~\cite[Section~6]{MatousekZ04}.
\end{remark}

\subsubsection{The Borsuk Graph}\label{sec:Borsuk}

For $0< \alpha <2$ and an integer $d$, the Borsuk graph $B(d,\alpha)$ is defined as the (infinite) graph on the vertex set $\S^{d-1}$ where two points $y,y' \in \S^{d-1}$ are adjacent if $\|y-y'\| \geq \alpha$.
It is known that the Borsuk-Ulam theorem implies that $\chi(B(d,\alpha)) \geq d+1$ for every $0< \alpha <2$ and $d$, and that this bound is tight whenever $\alpha \geq \sqrt{2(d+1)/d}$ (see, e.g.,~\cite{Lovasz83}).
We apply here the proof technique of Theorem~\ref{thm:geo} to obtain the same bound on the orthogonality dimension over $\R$ of the complement of $B(d,\alpha)$. We first prove the following.

\begin{theorem}\label{thm:BorsukGraph}
For $0< \eps <1$ and an integer $d$, let $V \subseteq \S^{d-1}$ be a finite collection of points in $\S^{d-1}$ such that for every $x \in \S^{d-1}$ there exists $y \in V$ for which $\|x-y\| < \eps$. Let $G$ be the graph on the vertex set $V$ where two points $y,y' \in V$ are adjacent if $\|y-y'\| \geq 2-2\eps$. Then, $\xi_\R(\overline{G}) \geq d+1$.
\end{theorem}

\begin{proof}
For a graph $G$ as in the theorem, put $t = \xi_\R(\overline{G})$.
Then there exists an assignment of a nonzero vector $u_y \in \R^t$ to every point $y \in V$, such that $\langle u_y,u_{y'} \rangle = 0$ for every points $y,y' \in V$ satisfying $\|y-y'\| \geq 2-2\eps$. It can be assumed without loss of generality that the first nonzero coordinate in every vector $u_y$ is positive (otherwise replace $u_y$ by $-u_y$).

Define a function $f:\S^{d-1} \rightarrow \R^t$ by
\[ f(x) = \sum_{y \in V}{u_y \cdot \max ( \eps - \|x-y\| ,0 )}.\]
For an $x \in \S^{d-1}$, let $C_x \subseteq V$ be the set of all points $y \in V$ such that $\|x-y\| < \eps$.
By the assumption on $V$ we have $|C_x| \geq 1$.
Observe that $f(x)$ is a linear combination with positive coefficients of the vectors $u_y$ with $y \in C_x$.
Letting $j \in [t]$ be the least coordinate in which at least one of the vectors of $\{ u_{y} \mid  y \in C_x \}$ is nonzero, using the assumption that the first nonzero coordinate of every $u_{y}$ is positive, it follows that the $j$th coordinate of $f(x)$ is positive, hence $f(x)$ is nonzero.
Further, since the distance between $x$ and $-x$ is $2$, for every $y \in C_x$ and $y' \in C_{-x}$ we have $\|y-y'\| \geq 2-2\eps$, and thus $\langle u_y, u_{y'} \rangle = 0$. This implies that the vectors $f(x)$ and $f(-x)$ are orthogonal.

Consider the function $g:\S^{d-1} \rightarrow \S^{t-1}$ defined by
$g(x) = \frac{f(x)}{\|f(x)\|}$.
Note that $g$ is well defined as $f(x)$ is nonzero for every $x \in \S^{d-1}$. Consider also the function $\widetilde{g} : \S^{d-1} \rightarrow \R^{t-1}$ that maps every $x \in \S^{d-1}$ to the projection of $g(x)$ to its last $t-1$ coordinates. We claim that there is no $x \in \S^{d-1}$ such that $\widetilde{g}(x)=\widetilde{g}(-x)$. To see this, notice that $g(x)$ is a unit vector whose first entry is non-negative, so the projection of $g(x)$ to its last $t-1$ coordinates fully determines $g(x)$.
This implies that if there exists an $x \in \S^{d-1}$ satisfying $\widetilde{g}(x)=\widetilde{g}(-x)$ then this $x$ also satisfies $g(x)=g(-x)$, in contradiction to the orthogonality of $f(x)$ and $f(-x)$. Since $\widetilde{g}$ is continuous we can apply Theorem~\ref{thm:BU} to derive that $t-1 > d-1$ which implies that $t \geq d+1$ and completes the proof.
\end{proof}

\begin{corollary}\label{cor:BorsukGraph}
For every $0< \alpha <2$ and an integer $d$, $\xi_\R(\overline{B(d,\alpha)}) \geq d+1$.
\end{corollary}

\begin{proof}
For $0< \alpha <2$ and an integer $d$, let $V$ be a maximal collection of points in $\S^{d-1}$ with pairwise distances at least $\eps = 1-\frac{\alpha}{2}$. Observe that $V$ is finite and that for every $x \in \S^{d-1}$ there exists $y \in V$ for which $\|x-y\| < \eps$. The graph $G$ associated with $V$ and $\eps$ in Theorem~\ref{thm:BorsukGraph} is a subgraph of $B(d,\alpha)$, hence $\xi_\R(\overline{B(d,\alpha)}) \geq \xi_\R(\overline{G}) \geq d+1$, and we are done.
\end{proof}

\section{Fractional Upper Bounds on the Shannon Capacity}\label{sec:Shannon4}

Let $f$ be a real-valued non-negative function on graphs. As explained in Section~\ref{sec:Shannon}, if $f$ is a sub-multiplicative upper bound on the independence number then it forms an upper bound on the Shannon capacity.
A fractional variant $f^*$ of $f$ was introduced in~\cite{HuTS17}, and it was shown there that if $f$ is a sub-multiplicative upper bound on the independence number then so is $f^*$, hence $f^*$ also forms an upper bound on the Shannon capacity. For a graph $G$ on the vertex $V$, the definition of $f^*(G)$ is given in~\eqref{Primal}, and by duality it is equal to the value of the following linear program.
\begin{align}
  \label{Dual}
  \begin{split}
 \textup{minimize } &\sum_{S \subseteq V} q(S) \cdot f(G[S]) \\
 \textup{subject to }  &\sum_{S: x \in S}q(S) \ge 1 \textup{~~~ for each vertex $x \in V$},\\
                &q(S)\ge 0 \textup{~~~ for each set $S \subseteq V$}.
              \end{split}
\end{align}
\noindent
Note that every graph $G$ satisfies $f^*(G) \leq f(G)$, as follows by taking $q(S)=1$ for $S=V$ and $q(S)=0$ otherwise.
We study here the integrality gap of fractional upper bounds $f^*$ on the Shannon capacity, namely, the largest possible ratio $f(G)/f^*(G)$ over all $n$-vertex graphs $G$.

\subsection{Upper Bound}

We prove now Theorem~\ref{thm:fractional_upper}, which claims that for sub-additive functions $f$ on graphs, the integrality gap of $f^*$ is at most logarithmic in the number of vertices. Recall that $f$ is sub-additive if for every graph $G$ on the vertex set $V$ and every sets $S_1$ and $S_2$ such that $V = S_1 \cup S_2$, it satisfies $f(G) \leq f(G[S_1]) + f(G[S_2])$.

\begin{proof}[ of Theorem~\ref{thm:fractional_upper}]
We use the dual definition of $f^*$ given in~\eqref{Dual}.
Let $G=(V,E)$ be an $n$-vertex graph, and let $q$ be an optimal solution of~\eqref{Dual}. Denote $Q = \sum_{S \subseteq V}{q(S)} \geq 1$, and let $D$ be the distribution over the subsets of $V$ that assigns to every set $S \subseteq V$ the probability $\frac{q(S)}{Q}$.
For, say, $t = \lceil Q \cdot \ln (3n) \rceil$, let $S_1,\ldots,S_t$ be $t$ random subsets of $V$ chosen independently from the distribution $D$.

We first claim that the probability that the sets $S_1,\ldots,S_t$ do not cover the entire vertex set $V$ is at most $1/3$.
Indeed, for every $x \in V$ and $i \in [t]$ the probability that $x \in S_i$ is
\[ \sum_{S:x \in S}{\frac{q(S)}{Q}} = \frac{1}{Q} \cdot \sum_{S:x \in S}{q(S)} \geq \frac{1}{Q},\]
where the inequality holds since $q$ is a feasible solution of~\eqref{Dual}.
Hence, the probability that $x$ does not belong to any of the sets $S_1,\ldots,S_t$ is at most
\[\Big (1-\frac{1}{Q} \Big )^t \leq e^{-t/Q} \leq \frac{1}{3n}.\]
By the union bound, the probability that there exists $x \in V$ such that $x \notin \cup_{i=1}^{t}{S_i}$ is at most $1/3$.

We next consider the expectation of the sum $\sum_{i=1}^{t}{f(G[S_i])}$.
For every $i \in [t]$ we have
\[ \expectation [f(G[S_i])] = \sum_{S \subseteq V} {\frac{q(S)}{Q} \cdot f(G[S])} = \frac{f^*(G)}{Q},\]
where the second equality holds since $q$ is an optimal solution of~\eqref{Dual}.
By linearity of expectation,
\[ \expectation \Big [\sum_{i=1}^{t}{f(G[S_i])} \Big ] = t \cdot \frac{f^*(G)}{Q} = \lceil Q \cdot \ln (3n) \rceil \cdot \frac{f^*(G)}{Q}  \leq 2 \cdot \ln (3n) \cdot f^*(G).\]
By Markov's inequality, the probability that $\sum_{i=1}^{t}{f(G[S_i])} \geq 6 \cdot \ln (3n) \cdot f^*(G)$ is at most $1/3$.

Finally, by the union bound, there are sets $S_1,\ldots,S_t$ that cover the vertex set $V$ and satisfy $\sum_{i=1}^{t}{f(G[S_i])} \leq 6 \cdot \ln (3n) \cdot f^*(G)$. For these sets the sub-additivity of $f$ implies that
\[f(G) \leq \sum_{i=1}^{t}{f(G[S_i])} \leq 6 \cdot \ln (3n) \cdot f^*(G),\] and we are done.
\end{proof}

\subsection{Lower Bounds}

We turn to show  that the ratio between $f$ and $f^*$ is unbounded for several upper bounds $f$ on the Shannon capacity. In fact, we present an explicit family of graphs for which $f^*$ is bounded from above by a constant whereas $f$ is arbitrarily large.
We consider here the graph parameters $\xi_\Fset$ and ${\minrank}_\Fset$, which, as mentioned before, are sub-additive and sub-multiplicative upper bounds on the independence number for every field $\Fset$.

We need the well-studied notion of fractional chromatic number of graphs. For a graph $G$ on the vertex set $V$, let $\calI(G)$ denote the collection of all independent sets of $G$. The {\em fractional chromatic number} of $G$, denoted by $\chi_f(G)$, is defined as the value of the following linear program.
\begin{align}
  \label{fractional_chi}
  \begin{split}
 \textup{minimize } &\sum_{I \in \calI(G)} q(I) \\
 \textup{subject to }  &\sum_{I \in \calI(G): x \in I}q(I) \ge 1 \textup{~~~ for each vertex $x \in V$},\\
                &q(I)\ge 0 \textup{~~~ for each set $I \in \calI(G)$}.
              \end{split}
\end{align}
\noindent
The following claim, given in~\cite{HuTS17}, follows directly from~\eqref{Dual} and~\eqref{fractional_chi}.
\begin{claim}\label{claim:f*chi_f}
Let $f$ be a function on graphs satisfying $f(G)=1$ whenever $G$ is complete. Then for every graph $G$, $f^*(G) \leq \chi_f(\overline{G})$.
In particular, for every graph $G$ and every field $\Fset$,
\[{\minrank}_\Fset^*(G) \leq \xi_\Fset^*(G) \leq \chi_f(\overline{G}).\]
\end{claim}

For integers $d \geq 2s$, consider the Kneser graph $K(d,s)$ (see Section~\ref{sec:Kneser}).
The number of vertices in $K(d,s)$ is $\binom{d}{s}$. By the Erd{\H{o}}s-Ko-Rado theorem~\cite{ErdosKoRado61} its independence number is $\alpha(K(d,s)) = \binom{d-1}{s-1}$, and as already mentioned, its chromatic number is $\chi(K(d,s)) = d-2s+2$~\cite{LovaszKneser}.
It is known that every vertex-transitive graph $G=(V,E)$ (that is, a graph whose automorphism group is transitive), satisfies $\chi_f(G) = \frac{|V|}{\alpha(G)}$ (see, e.g.,~\cite{GraphTheoryBook}). Hence,
\begin{eqnarray}\label{eq:fracKneser}
\chi_f(K(d,s)) = \frac{\binom{d}{s}}{\binom{d-1}{s-1}} = \frac{d}{s}.
\end{eqnarray}

Now we are ready to derive the following theorem, which confirms Theorem~\ref{thm:IntroIntegrality}.
\begin{theorem}\label{thm:AllBounds}
For every fixed $\eps >0$, there exists an explicit family of $n$-vertex graphs $G$ such that
\begin{enumerate}
  \item\label{itm:1} ${\minrank}^*_\Fset(G) \leq \xi^*_\Fset(G) \leq 2+\eps$ ~~for every field $\Fset$,
  \item\label{itm:2} $\xi_\R (G)= \Theta(\log n)$,
  \item\label{itm:3} $\xi_\C (G)= \Theta(\log n)$,
  \item\label{itm:4} ${\minrank}_\R (G) \geq \Omega( \sqrt{ \log n})$, and
  \item\label{itm:5} $\xi_\Fset(G) \geq {\minrank}_\Fset (G) \geq \Omega( \log \log n)$ ~~for every fixed finite field $\Fset$.
\end{enumerate}
\end{theorem}

\begin{proof}
For a fixed $\eps>0$, let $d$ and $s$ be two integers satisfying $d = (2+\eps) \cdot s$. Let $G$ be the complement of the Kneser graph $K(d,s)$.
We show that $G$ satisfies the assertion of the theorem.
By~\eqref{eq:fracKneser}, we have
\[\chi_f(\overline{G}) = \chi_f(K(d,s)) = \frac{d}{s} = 2+\eps,\]
hence Item~\ref{itm:1} of the theorem follows from Claim~\ref{claim:f*chi_f}.
The graph $G$ has $n = \binom{d}{s} = 2^{\Theta(s)}$ vertices, so we have $d-2s+2 = \eps \cdot s +2 = \Theta(\log n)$. For the upper bounds in Items~\ref{itm:2} and~\ref{itm:3} notice that for every field $\Fset$, $\xi_\Fset(G) \leq \chi(\overline{G}) = d-2s+2 = \Theta(\log n)$.
The lower bounds in Items~\ref{itm:2},~\ref{itm:3}, and~\ref{itm:4} follow from Corollary~\ref{cor:xiKneser}.
Finally, Item~\ref{itm:5} follows from Claim~\ref{claim:minrankF}.
\end{proof}

\section*{Acknowledgments}
We would like to thank Itzhak Tamo for useful conversations and for his comments on an early version of the paper.
We are also grateful to Florian Frick and G\"unter M. Ziegler for helpful discussions.

%\newpage
\bibliographystyle{abbrv}
\bibliography{kneser}

\end{document}